\documentclass[11pt]{amsart}
\usepackage{graphicx}
\usepackage{amsmath, amssymb}
\usepackage{verbatim}
\usepackage{color}
\newtheorem{thm}{Theorem}[section]
\newtheorem{cor}[thm]{Corollary}
\newtheorem{conj.}[thm]{Conjecture}
\newtheorem{lem}[thm]{Lemma}
\newtheorem{prop}[thm]{Proposition}
\theoremstyle{definition}
\newtheorem{defn}[thm]{Definition}
\theoremstyle{remark}

\numberwithin{equation}{section}
\newtheorem{exa}[thm]{Example}

\newcommand{\h}{\mathcal{H}}

\begin{document}

\title[Dual pair and Approximate dual for continuous frames ...]
{Dual pair and Approximate dual for continuous frames in Hilbert spaces }%
\author[A. Rahimi, Z. Darvishi, B. Daraby ]{A. Rahimi, Z. Darvishi, B. Daraby}
\address{Department of Mathematics, University of Maragheh, Maragheh, Iran.}
\email{rahimi@maragheh.ac.ir}
\address{Department of Mathematics, University of Maragheh, Maragheh, Iran.}
\email{darvishi\_z@ymail.com}
\address{Department of Mathematics, University of Maragheh, Maragheh, Iran.}
\email{bdaraby@maragheh.ac.ir}


\begin{abstract}
In this manuscript, the concept of dual  and approximate dual for continuous frames in Hilbert spaces will be introduced. Some of its properties will be studied. Also, the relations between two continuous Riesz bases  in Hilbert spaces will be clarified through examples.
\end{abstract}
\maketitle
\section{Introduction}
Frames for Hilbert spaces have been first introduced by Duffin and Scheaffer
in the study of some problems in nonharmonic Fourier series in 1952, \cite{duf}.
A discrete frame is a countable family of elements in a separable Hilbert space
which allows for a stable, not necessarily unique,
decomposition of an arbitrary element into an expansion of the frame elements.
\par Recall that for a Hilbert space $\h$ and a countable index set $I$,
a family of vectors $\{f_i\}_{i\in I}\subseteq\h$ is called a discrete frame for $\h$,
if there exist constants $0<A\leq B<+\infty$ such that
$$A\|f\|^2\leq\sum_{i\in I}|\langle f,f_i\rangle|^2\leq B\|f\|^2,\hspace{.5cm}f\in\h,$$
the constants $A$ and $B$ are called frame bounds. The frame $\{f_i\}_{i\in I}$ is called tight if $A=B$ and Parseval if $A=B=1$.
The frame decomposition is the most important frame results.
It shows for the frame $\{f_i\}_{i\in I}$,
every element in $\h$ has a representation as an infinite linear combination of the frame elements;
 i.e., there exist coefficients $\{c_i(f)\}_{i\in I}$ such that
$f=\sum_{i\in I}c_i(f)f_i$, where $f\in\h$ is arbitrary.
Thus, it is natural to say that a frame is some kind of a generalized basis.
Usually, we want to work with coefficients which depend
continuously and linearly on $f$,
by ``Riesz representation theorem'',
this implies that the $i$-th coefficient in the expansion of $f$
should have on the form $c_i(f)=\langle f,g_i\rangle$ for some $g_i\in\h$.
The sequence $\{g_i\}_{i\in I}$ is called a dual frame of $\{f_i\}_{i\in I}$.
Li \cite{li} provided a characterization and construction of general frame decomposition. He showed that for generating all duals for a given frame,
it is enough to find the
left-inverses of a one-to-one mapping
and found a general parametric and algebraic formula for all duals.
It is usually complicated to calculate a dual frame explicitly. Hence Christensen and Laugesen seek methods for constructing approximate duals,
see \cite{olela}.
Note that the idea of approximate dual frames has appeared previously,
especially for wavelets \cite{hol}, Gabor systems \cite{balfe,feka},
in the general context of coorbit theory \cite{fegr}
and the sensor modeling \cite{liya}.
\par New applications of frames, specially in the last decade, motivated the researcher to find some generalizations of frames like continuous frames \cite{ali,dr2}, fusion frames \cite{6}, $g$-frames \cite{sun}, controlled and weighted frames \cite{Balazs,rafe}, $p$-frames \cite{18}, $K$-frames \cite{Gavruta} and etc.\par
The notion of continuous frames was introduced by Kaiser in \cite{kai}
and independently by Ali, Antoine and Gazeau \cite{ali}.
The windowed Fourier transform and the continuous wavelet transform are just two instances of continuous frames but at the same time the main motivation for its definition.
Gabardo and Han in \cite{han} defined the concept of  dual frames for continuous frames.

\section{Preliminaries}
In this section, we review some notations and definitions.
Throughout this paper, $\h$ is a Hilbert space
and $(\Omega,\mu)$ a measure space with positive measure $\mu$.
\begin{defn}
A weakly-measurable mapping $F:\Omega\rightarrow\h$ is called a
continuous frame for $\h$ with respect to $(\Omega,\mu)$
if there exist constants $0<A\leq B<\infty$ such that
$$A\|f\|^2\leq\int_{\Omega}|\langle f,F(\omega)\rangle|^2 d\mu(\omega)\leq B\|f\|^2,\hspace{.5cm}f\in\h.$$
The constants $A$ and $B$ are called continuous frame bounds.
This mapping $F$ is called tight continuous frame if $A=B$
and if $A=B=1$, it called a Parseval continuous frame.
This mapping is called Bessel if the second inequality holds.
In this case, $B$ is called Bessel constant.\par
Suppose $F:\Omega\rightarrow\h$ is a Bessel mapping with bound $B$.
The operator $T_F:L^2(\Omega,\mu)\rightarrow\h$ weakly defined by
$$\langle T_F\varphi,f\rangle=\int_{\Omega}\varphi(\omega)\langle F(\omega),f\rangle d\mu(\omega),
\hspace{.5cm}\varphi\in L^2(\Omega,\mu),\hspace{.5cm}f\in\h,$$
is well-defined, linear, bounded with bound $\sqrt{B}$ and
its adjoint is given by
$$T_F^*:\h\rightarrow L^2(\Omega,\mu),\hspace{.5cm}(T_F^*f)(\omega)=\langle f,F(\omega)\rangle,
\hspace{.5cm}\omega\in\Omega,\hspace{.5cm}f\in\h.$$
The operator $T_F$ is called the synthesis operator and
$T_F^*$ is called the analysis operator of $F$.
For continuous frame $F$ with bounded $A$ and $B$,
the operator $S_F=T_FT_F^*$ is called continuous frame operator
and this is bounded, invertible, positive and $AI_{\h}\leq S_F\leq BI_{\h}$.
In fact,
$$\langle S_Ff,g\rangle=\int_{\Omega}\langle f,F(\omega)\rangle\langle F(\omega),g\rangle d\mu(\omega),
\hspace{.5cm}f,g\in\h.$$
For all $f,g\in\h$, the reconstruction formulas are as follows
$$\langle f,g\rangle=\int_{\Omega}
\langle f,F(\omega)\rangle\langle S_F^{-1}F(\omega),g\rangle d\mu(\omega)
=\int_{\Omega}
\langle f,S_F^{-1}F(\omega)\rangle\langle F(\omega),g\rangle d\mu(\omega).$$
\end{defn}
We end this short section by a well known example named wavelet frames.
\begin{exa}\label{2.2}
Let $\Omega=(0,+\infty)\times\mathbb R$ be the affine group,
with group law $(a,b)(a',b')=(aa',b+ab')$.
An element $\psi\in L^2(\mathbb R)$ is said to be admissible if
$\|\psi\|_2=1$ and $C_{\psi}=\int_{0}^{+\infty}\frac{|\hat{\psi}(\xi)|^2}{\xi}d\xi<+\infty$.
For such admissible function $\psi$,  we have
$$\int_{-\infty}^{+\infty}\int_{0}^{+\infty}|\langle f,T_aD_b\psi\rangle|^2\frac{dadb}{a^2}=C_{\psi}\|f\|^2,\hspace{.5cm} f\in L^2(\mathbb R),$$
where $T_af(t)=f(t-a)$ and $D_bf(t)=\frac{1}{\sqrt{b}}f(\frac{t}{b})$, see \cite{chri}.
That is, $\{T_aD_b\psi\}_{a>0,b\in\mathbb R}$ is a
tight continuous frame with respect to $(\Omega,\frac{dadb}{a^2})$
with the frame operator $S=C_{\psi}I$.
\end{exa}
\section{The duality of continuous frames}
Reconstruction of the original vector from frames,
g-frames, fusion frames, continuous frames as well as their extensions,
is typically achieved by using a so-called (alternate or standard) dual system.
\begin{defn}
Let $F$  and $G$ be two Bessel mappings
with synthesis operators $T_F$ and $T_G$, respectively.
We call $G$ a dual of $F$ if the following equality holds
$$\langle f,g\rangle=\int_{\Omega}\langle f,F(\omega)\rangle\langle G(\omega),g\rangle d\mu(\omega),\hspace{.5cm}f,g\in\h.$$
In this case, $(F,G)$ is called a dual pair for $\h$.
\end{defn}
This definition is equivalent to $T_GT_F^*=I_{\h}$.
The condition
$$\langle f,g\rangle=\int_{\Omega}\langle f,F(\omega)\rangle\langle G(\omega),g\rangle d\mu(\omega),\hspace{.5cm}f,g\in\h$$
is equivalent
$$\langle f,g\rangle=\int_{\Omega}\langle f,G(\omega)\rangle\langle F(\omega),g\rangle d\mu(\omega),\hspace{.5cm}f,g\in\h$$
because $T_GT_F^*=I$ if and only if $T_FT_G^*=I$.\par
For the continuous frame $F$, the mapping $S_F^{-1}F$ is called standard dual of $F$.
It is certainly possible for a continuous frame $F$ to have only one dual.
In this case, the continuous frame $F$ is called a Riesz-type frame.
Riesz-type frames are actually frames,
for which the analysis operator is onto. Also, the continuous frame $F$
 is Riesz-type frame if and only if $T_F^*$ is onto \cite{arefi}.
 \par
Similar to discrete case, by a simple calculation, it is easy to show that $G$ is a dual of $F$ if and only if
$G=S_F^{-1}F+H$, where $H$ satisfies in
the condition $\int_{\Omega}\langle f,F(\omega)\rangle \langle H(\omega),g\rangle d\mu(\omega)=0$, for all $f,g\in\h$.\\
It is easy verifiable that
$F$ is a tight continuous frame for $\h$ with bound $C$ if and only if $(F,\frac{1}{C}F)$ is dual pair for it. \par
By using Example \ref{2.2},
the pair $(T_aD_b\psi,\frac{1}{C_{\psi}}T_aD_b\psi)$ is a dual pair in $L^2(\mathbb R)$ and we have
$$f=\frac{1}{C_{\psi}}\int_{-\infty}^{+\infty}\int_{0}^{+\infty}\langle f,T_aD_b\psi\rangle T_aD_b\psi\frac{dadb}{a^2},\hspace{.5cm}f\in L^2(\mathbb R).$$
It is not clear for which wavelet frames the standard dual frame
consists of wavelet as well.
More generally, there are some wavelet frames with no dual wavelet frames at all \cite{da}.\par
The following is an another example of continuous frame with one of its duals.
\begin{exa}\label{3.3}
Consider $\h=\mathbb{R}^2$ with the standard basis $\{e_1,e_2\}$,
where $e_1=(1,0)$ and $e_2=(0,1)$.
Put $B_{\mathbb{R}^2}=\{x\in\mathbb{R}^2:~~\|x\|\leq1\}$.
Let $\Omega=B_{\mathbb{R}^2}$ and $\lambda$ be the Lebesgue measure.
Define $F:B_{\mathbb{R}^2}\to\mathbb{R}^2$ and $G:B_{\mathbb{R}^2}\to\mathbb{R}^2$ such that
\[
F(\omega)=\left\{\begin{array}{ll}
\frac{1}{\sqrt{\lambda(B_1)}}e_1,\hspace{.5cm}\omega\in B_1,\cr
\frac{1}{\sqrt{\lambda(B_2)}}e_2,\hspace{.5cm}\omega\in B_2,\cr
0,\hspace{1.9cm}\omega\in B_3,\cr
\end{array}\right.
\]
and
\[
G(\omega)=\left\{\begin{array}{ll}
\frac{1}{\sqrt{\lambda(B_1)}}e_1,\hspace{.9cm}\omega\in B_1,\cr
\frac{1}{\sqrt{\lambda(B_2)}}e_2,\hspace{.9cm}\omega\in B_2,\cr
\frac{2}{\sqrt{\lambda(B_3)}}e_2,\hspace{.9cm}\omega\in B_3,\cr
\end{array}\right.
\]
where $\{B_1,B_2,B_3\}$ is a partition of $B_{\mathbb{R}^2}$.
It is easy to check that $F$ and $G$ are continuous frames for $\mathbb{R}^2$ with respect to $(B_{\mathbb{R}^2},\lambda)$.
For each $x\in\mathbb{R}^2$, we have
\begin{align*}
(T_GT_F^*)(x)
& = \left(\int_{B_1}+\int_{B_2}+\int_{B_3}\right)(T_F^*x)(\omega)G(\omega)d\lambda(\omega)\\
& = \int_{B_1}\langle x,\frac{1}{\sqrt{\lambda(B_1)}}e_1\rangle\frac{1}{\sqrt{\lambda(B_1)}}e_1d\lambda(\omega)\\
&\hspace{.5cm} + \int_{B_2}\langle x,\frac{1}{\sqrt{\lambda(B_2)}}e_2\rangle\frac{1}{\sqrt{\lambda(B_2)}}e_2d\lambda(\omega) \\
&\hspace{.5cm} + 0\\
& = \langle x,e_1\rangle e_1 + \langle x,e_2\rangle e_2\\
& = x
\end{align*}
i.e., $(F,G)$ is a dual pair for $\mathbb{R}^2$.
\end{exa}
Similar to discrete frames, for continuous frame we have the following assertion.
\begin{prop}
The Bessel mapping $F:\Omega\to\h$ is a continuous frame for $\h$ with respect to $(\Omega,\mu)$ if and only if
there exists a Bessel mapping $G:\Omega\to\h$ such that for each $f,g\in\h$,
\begin{align}
\langle f,g\rangle=\int_{\Omega}\langle f,F(\omega)\rangle \langle G(\omega),g\rangle d\mu(\omega).
\end{align}
\end{prop}
\begin{proof}
At the first, assume that there exists a Bessel mapping
$G$ with Bessel constant $B_G$, satisfying (3.1).
Then for any $f\in\h$,
\begin{align*}
\|f\|^4
&= \left|\int_{\Omega}\langle f,F(\omega)\rangle\langle G(\omega),f\rangle d\mu(\omega) \right|^2\\
&\leq\left(\int_{\Omega}|\langle f,F(\omega)\rangle\langle G(\omega),f\rangle| d\mu(\omega)\right)^2\\
&\leq\left(\int_{\Omega}|\langle f,F(\omega)\rangle|^2 d\mu(\omega)\right)
\left(\int_{\Omega}|\langle f,G(\omega)\rangle|^2 d\mu(\omega)\right)\\
&\leq\left(\int_{\Omega}|\langle f,F(\omega)\rangle|^2 d\mu(\omega)\right). B_G\|f\|^2.\\
\end{align*}
Thus, $F$ is a continuous frame for $\h$ with lower bound $B_G^{-1}$.\\
Conversely, let $F$ be a continuous frame for $\h$ with the frame operator $S_F$.
Thus, for all $f,g\in\h$
$$\langle f,g\rangle=\int_{\Omega}\langle f,F(\omega)\rangle \langle S_F^{-1}F(\omega),g\rangle d\mu(\omega).$$
Put $G=S_F^{-1}F$.
\end{proof}
Like discrete frames \cite{chri}, for Bessel mappings,
we have the following corollary.
\begin{cor}
If $(F,G)$ is a dual pair for $\h$, then both $F$ and $G$ are continuous frames for $\h$.
\end{cor}
Improving and extending Theorem 3.6 of \cite{dr2}
for standard dual,
we have the following theorem for any dual pairs.
It shows that when, we can remove some elements from
a continuous frame so that the remaining set is still a continuous frame.
\begin{thm}
Let $(F,G)$ be a dual pair for $\h$ and
there exists $\omega_0\in\Omega$ such that
$\mu(\{\omega_0\})\langle F(\omega_0),G(\omega_0)\rangle\neq1$.
Then $F:\Omega\backslash\{\omega_0\}\rightarrow \h$
is a continuous frame for $\h$.
\end{thm}
\begin{proof}
We assume that the Bessel constants $F$ and $G$ are $B_F$ and $B_G$, (respectively).
If $f\in\h$, then
\begin{align*}
\langle F(\omega_0),f\rangle
&= \int_{\Omega\backslash\{\omega_0\}}\langle F(\omega_0),G(\omega)\rangle\langle F(\omega),f\rangle d\mu(\omega)\\
&\hspace{.3cm}+ \langle F(\omega_0),G(\omega_0)\rangle\langle F(\omega_0),f\rangle \mu(\{\omega_0\}).\\
\end{align*}
Therefore,
\begin{align*}
|\langle f,F(\omega_0)\rangle |^2
& \leq\frac{1}{|1-\mu(\{\omega_0\})\langle F(\omega_0),G(\omega_0)\rangle |^2}\left(\int_{\Omega\backslash\{\omega_0\}}|\langle f,F(\omega)\rangle|^2 d\mu(\omega)\right)\\
&\hspace{.5cm} \left(\int_{\Omega\backslash\{\omega_0\}}|\langle F(\omega_0),G(\omega)\rangle |^2d\mu(\omega)\right)\\
& \leq\frac{B_G\|F(\omega_0)\|^2}{|1-\mu(\{\omega_0\})\langle F(\omega_0),G(\omega_0)\rangle |^2}\left(\int_{\Omega\backslash\{\omega_0\}}|\langle f,F(\omega)\rangle |^2d\mu(\omega)\right).\\
\end{align*}
Put $C=\frac{B_G\|F(\omega_0)\|^2}{|1-\mu(\{\omega_0\})\langle F(\omega_0),G(\omega_0)\rangle |^2}$.
We have
$$B_G^{-1}\|f\|^2\leq(1+C\mu(\{\omega_0\}))\int_{\Omega\backslash\{\omega_0\}}|\langle f,F(\omega)\rangle |^2d\mu(\omega).$$
Hence $F:\Omega\backslash\{\omega_0\}\to\h$ is a continuous frame with lower bound $\frac{B_G^{-1}}{1+C\mu(\{\omega_0\})}.$
\end{proof}
Concerning the above theorem, a question occurs. What happen in case $\mu(\{\omega_0\})\langle F(\omega_0),S_F^{-1}F(\omega_0)\rangle=1$?
To achieve this purpose, we need the following lemma.

 \begin{lem}\cite{dr2}\label{3.7}
The continuous frame coefficients
$\{\langle f,S_F^{-1}F(\omega)\rangle\}_{\omega\in\Omega}$
has minimal $L^2$-norm among all coefficients
$\{\phi(\omega)\}_{\omega\in\Omega}$
for which
$$f=\int_{\Omega}\phi(\omega)F(\omega)d\mu(\omega)$$
for some $\phi\in L^2(\Omega,\mu)$; i. e.,
$$\int_{\Omega}|\phi(\omega)|^2d\mu(\omega)=\int_{\Omega}|\langle f,S_F^{-1}F(\omega)\rangle|^2d\mu(\omega)+\int_{\Omega}|\phi(\omega)-\langle f,S_F^{-1}F(\omega)\rangle|^2d\mu(\omega).$$
\end{lem}
\begin{thm}
Let $F$ be a continuous frame for $\h$ with respect to $(\Omega,\mu)$
with frame operator $S_F$
 and $\omega_0\in\Omega$.
If $\mu(\{\omega_0\})\langle F(\omega_0),S_F^{-1}F(\omega_0)\rangle=1$,
then there exists a non-empty measurable set $\Omega_0\subset\Omega$
such that   $\omega_0\in\Omega_0$,
$\mu(\Omega_0)=0$
and $\{F(\omega)\}_{\omega\in\Omega\backslash\Omega_0}$ is incomplete.
\end{thm}
\begin{proof}
We have
$$F(\omega_0) = \int_{\Omega}\langle F(\omega_0),S_F^{-1}F(\omega)\rangle F(\omega)d\mu(\omega) =
\int_{\Omega}\frac{\chi_{\{\omega_0\}}}{\mu(\{\omega_0\})}F(\omega) d\mu(\omega),$$
such that $\chi_{\{\omega_0\}}$ is
the characteristic function of a set ${\{\omega_0\}}\subset\Omega$.
So Lemma \ref{3.7} yields the following relation between
 $\{\frac{\chi_{\{\omega_0\}}(\omega)}{\mu(\{\omega_0\})}\}$
and $\{\langle F(\omega_0),S_F^{-1}F(\omega)\rangle\}$
\begin{align*}
\int_{\Omega}|\frac{\chi_{\{\omega_0\}}(\omega)}{\mu(\{\omega_0\})}|^2d\mu(\omega)
& =\int_{\Omega}|\langle F(\omega_0),S_F^{-1}F(\omega)\rangle|^2d\mu(\omega)\\
& \hspace{.5cm}+\int_{\Omega}|\frac{\chi_{\{\omega_0\}}(\omega)}{\mu(\{\omega_0\})}-\langle F(\omega_0),S_F^{-1}F(\omega)\rangle|^2d\mu(\omega)\\
& =\int_{\Omega\backslash\{\omega_0\}}|\langle F(\omega_0),S_F^{-1}F(\omega)\rangle|^2d\mu(\omega)\\
& \hspace{.5cm}+|\langle F(\omega_0),S_F^{-1}F(\omega_0)\rangle|^2\mu(\{\omega_0\})\\
& \hspace{.5cm}+\int_{\Omega\backslash\{\omega_0\}}|\frac{\chi_{\{\omega_0\}}(\omega)}{\mu(\{\omega_0\})}-\langle F(\omega_0),S_F^{-1}F(\omega)\rangle|^2d\mu(\omega)\\
& \hspace{.5cm}+|\frac{1}{\mu(\{\omega_0\})}-\langle F(\omega_0),S_F^{-1}F(\omega_0)\rangle|^2\mu(\omega_0).
\end{align*}
From the above formula,
$$\int_{\Omega\backslash\{\omega_0\}}|\langle F(\omega_0),S_F^{-1}F(\omega)\rangle|^2d\mu(\omega)=0,$$
so that $\langle F(\omega_0),S_F^{-1}F(\omega)\rangle=0$
a.e. on $\Omega\backslash\{\omega_0\}$.
Put
$$\Omega_0=\{\omega\in\Omega:~~\langle S_F^{-1}F(\omega_0),F(\omega)\rangle\neq0\}.$$
It is clear that $\Omega_0$ is a measurable set with zero measure and $\omega_0\in\Omega_0$.
Thus we have the non-zero element $S_F^{-1}F(\omega_0)$
which is orthogonal to $\{F(\omega)\}_{\omega\in\Omega\backslash\Omega_0}$,
i.e.,  $\{F(\omega)\}_{\omega\in\Omega\backslash\Omega_0}$ is incomplete.
\end{proof}
Now we are going to give simple ways for construction of many dual pairs of a given dual pair.
\begin{thm}
Let $(F,G)$ be a dual pair for $\h$ and
let $U$ and $V$ be two bounded operators on $\h$ such that $VU^*=I_{\h}$.
Then, $(UF,VG)$ is a dual pair for $\h$.
\end{thm}
\begin{proof}
It is clear that if $F$ is a Bessel mapping with synthesis operator $T_F$ and
$U$ is a bounded operator on $\h$,
then $UF$ is a Bessel mapping  with synthesis operator $T_{UF}=UT_F$.
Hence $UF$ and $VG$ are Bessel mappings and
$$T_{VG}T_{UF}^*=VT_GT_F^*U^*=VIU^*=I_{\h}.$$
 \end{proof}
\begin{cor}
If $(F,G)$ is a dual pair for $\h$ and $U$ is a unitary operator,
then $(UF,UG)$ is a dual pair for $\h$.
\end{cor}
\begin{thm}
Let $(F,G)$ be a dual pair for $\h$ and
there exists a bounded operator $U\in L(\h)$ such that $(F,UG)$ is a dual pair for $\h$.
Then $U=I_{\h}$.
\end{thm}
\begin{proof}
For all $f,g\in\h$,
$$\langle f,U^*g\rangle=\int_{\Omega}\langle f,F(\omega)\rangle\langle G(\omega),U^*g\rangle d\mu=\int_{\Omega}\langle f,F(\omega)\rangle\langle UG(\omega),g\rangle d\mu=\langle f,g\rangle.$$
Therefore, $U^*=I_{\h}$.
Hence $U=I_{\h}$.
\end{proof}
\begin{thm}
Assume that  $(F,G)$ and $(F,K)$ are dual pairs for $\h$.
Then for all $\alpha\in\mathbb{C}$, $(F,\alpha G+(1-\alpha)K)$ is a dual pair for $\h$.
\end{thm}
\begin{proof}
Put $F_1=\alpha G+(1-\alpha)K$. For all $f,g\in\h$, we have
\begin{align*}
\int_{\Omega}\langle f,F(\omega)\rangle
\langle F_1(\omega),g\rangle d\mu(\omega)
& =\int_{\Omega}\langle f,F(\omega)\rangle
\langle \alpha G(\omega)+(1-\alpha)K(\omega),g\rangle d\mu(\omega)\\
& =\alpha\int_{\Omega}\langle f,F(\omega)\rangle\langle G(\omega),g\rangle d\mu(\omega)\\
&\hspace{.5cm}+ (1-\alpha)\int_{\Omega}\langle f,F(\omega)\rangle\langle K(\omega),g\rangle d\mu(\omega)\\
& =\alpha\langle f,g\rangle+(1-\alpha)\langle f,g\rangle\\
& =\langle f,g\rangle.
\end{align*}
\end{proof}
Now, we want to find a relationship between the arbitrary continuous Riesz basis of $\h$.
For this purpose, we need to following definitions and propositions from \cite{arefi}.
Denote by $L^2(\Omega,\h)$
the set of all mapping $F : \Omega\rightarrow\h$ such that for all $f\in\h$,
the functions $\omega\mapsto\langle f,F(\omega)\rangle$ defined almost everywhere on $\Omega$,
belongs to $L^2(\Omega)$.
 \begin{defn} \cite{arefi}
A Bessel mapping $F : \Omega\rightarrow\h$ is called $\mu$-complete, if
$$cspan\{F(\omega)\}_{\omega\in\Omega}=
\left\{\int_{\Omega}\varphi(\omega)F(\omega)d\mu(\omega);\hspace{.5cm}\varphi\in L^2(\Omega)\right\}$$
 is dense in $\h$.
 \end{defn}
It is worthwhile to mention that if  $F$ is $\mu$-complete,
then  $ \{F(\omega)\}_{\omega\in\Omega}$  is a complete subset of $\h$.
The converse is also true when
 $0<\mu(\{\omega\})<+\infty$ for all $\omega\in\Omega$,
since $span\{F(\omega)\}_{\omega\in\Omega}\subseteq cspan\{F(\omega)\}_{\omega\in\Omega}.$
\begin{prop} \cite{arefi}
Let $F$ be a Bessel mapping.
The following are equivalent
\begin{enumerate}
  \item $F$ is $\mu$-complete;
  \item If $f\in\h$ so that $\langle f,F(\omega)\rangle=0$ for almost all $\omega\in\Omega$,
then $f=0$.
  \end{enumerate}
\end{prop}
\begin{defn} \cite{arefi}
A mapping $F\in L^2(\Omega,\h)$ is called a continuous Riesz base for $\h$
with respect to $(\Omega,\mu)$,
if $\{F(\omega)\}_{\omega\in\Omega}$ is $\mu$-complete and
there are two positive numbers $A$ and $B$ such that
$$A\left(\int_{\Omega_1}|\phi(\omega)|^2 d\mu(\omega)\right)^{\frac{1}{2}}\leq
\left\|\int_{\Omega_1}\phi(\omega)F(\omega)d\mu(\omega)\right\|
\leq B\left(\int_{\Omega_1}|\phi(\omega)|^2 d\mu(\omega)\right)^{\frac{1}{2}},$$
for every $\phi\in L^2(\Omega)$ and
any measurable subset $\Omega_1$ of $\Omega$ with $\mu(\Omega_1)<+\infty$.
The integral is taken in the weak sense and
the constant $A$ and $B$ are called continuous Riesz base bounds.
It is obvious that any continuous Riesz basis is a continuous frame.
\end{defn}
\begin{defn} \cite{arefi}
A Bessel mapping $F$ is said to be $L^2$-independent if\\
 $\int_{\Omega} \varphi(\omega)F(\omega)d\mu(\omega)=0$
for $\varphi\in L^2(\Omega,\mu)$, implies that $\varphi=0$ a. e.
\end{defn}
\begin{prop} \cite{arefi}
Let $F\in L^2(\Omega,\h)$ be a continuous frame for $\h$.
Then $F$ is a continuous Riesz base for $\h$ if and only if $F$ is $\mu$-complete and $L^2$-independent.
\end{prop}
\begin{prop} \cite{arefi}
Let $F\in L^2(\Omega,\h)$ be a continuous frame.
The following are equivalent
\begin{enumerate}
  \item $F$ is a continuous Riesz base for $\h$;
  \item $F$ is a  Riesz-type continuous frame for $\h$;
   \item $T^*_F$ is onto.
  \end{enumerate}
\end{prop}
\hspace{-.5cm}
\begin{thm}
Let $F$ and $G$ be two continuous Riesz bases for $\h$.
Then there exists an invertible operator $\Theta\in L(\h)$ such that $G=S_G\Theta^*F.$
 \end{thm}
 \begin{proof}
Assume $f\in\h$ such that $(T_GT_F^*)f=0$.
We have for all $\omega\in\Omega$,   $T_G(T_F^*f(\omega))=0$.
Hence $\int_{\Omega}\langle f,F(\omega)\rangle G(\omega) d\mu(\omega)=0$.
Since $G$ is $L^2$-independent,
then $\langle f,F(\omega)\rangle=0$ for almost all $\omega\in\Omega$.
The $\mu$-completeness of $F$ implies $f=0$.
Therefore $T_GT_F^*$ is one to one.
According to Proposition 3.17, $T_F^*$ is onto,
because $F$ is a continuous Riesz base.
The synthesis $T_G$ is onto because $G$ is a continuous frame \cite{dr2}.
Hence, $T_GT_F^*$ is onto.
Putting $\Theta=(T_GT_F^*)^{-1}$, for any $f,g\in\h$,
\begin{align*}
\langle f,g\rangle
& = \langle\Theta^{-1}\Theta f,g\rangle\\
& = \langle T_F^*\Theta f,T_G^*g\rangle\\
& = \int_{\Omega}\langle\Theta f,F(\omega)\rangle\langle G(\omega),g\rangle d\mu(\omega)\\
& = \int_{\Omega}\langle f,\Theta^*F(\omega)\rangle\langle G(\omega),g\rangle d\mu(\omega).
\end{align*}
It follows that $\Theta^*F$ is a dual of $G$,
but $G$ has only one dual.
Hence $S_G^{-1}G=\Theta^*F$, or $G=S_G\Theta^*F$.
\end{proof}
\begin{defn} \cite{arefi}
A continuous orthonormal basis for $\h$ with respect to $(\Omega,\mu)$ is a
continuous Parseval frame $F$ for which
 $$\left\|\int_{\Omega}\phi(\omega)F(\omega)d\mu(\omega)\right\|=\|\phi\|_2, \hspace{.3cm}\phi\in L^2(\Omega).$$
\end{defn}
One may easily see that if $F$ is a continuous orthonormal bases for $\h$,
then it is a continuous Riesz base.
\begin{cor}
If $F$ and $G$ are two continuous orthonorml basis for $\h$,
then there exists an invertible operator $\Theta\in L(\h)$ such that $G=\Theta^*F.$
\end{cor}
\section{Approximate duality of continuous frames}
In Section 3, by definition of the dual frame, we saw that
if the Bessel mapping $G$ is a dual of Bessel mapping $F$,
then for all arbitrary elements $f,g\in\h$,
we have the dual frame expansion
$\langle f,g\rangle=\int_{\Omega}\langle f,F(\omega)\rangle\langle G(\omega),g\rangle d\mu(\omega)$.
Unfortunately, it might be difficult, or even impossible,
to calculate a dual frame explicitly.
This limitation leads us to seek continuous frames that are "close to dual".
For solving this problem in discrete frames,  Christensen and Laugesen in \cite{olela} introduced the concept of \textbf{approximately dual frames}. By using their ideas in this section, we investigate and improve this notion for continuous frames. Here, we are generalizing the definition 3.1. of \cite{olela} to continuous cases and then we will obtain  our results. For clarifying some examples will presented.\par
Throughout this section, we assume that $F$ and $G$ are Bessel mappings
with synthesis operators $T_F$ and $T_G$, respectively.
\begin{defn}
Two Bessel mappings $F$ and $G$ are called  approximately dual continuous frames for $\h$ if
$\|I_{\h}-T_GT_F^*\|<1$ or $\|I_{\h}-T_FT_G^*\|<1$.
\end{defn}
 It is clear that in this case, $T_GT_F^*$ is an invertible operator.
\begin{exa}
Similar to Example \ref{3.3}, let $\h=\mathbb{R}^2$ and $\{e_1,e_2\}$ be standard base for it.
Also, let $0<\varepsilon<1$ be arbitrary.
Put $\Omega=B_{\mathbb{R}^2}$ and $\lambda$ is the Lebesgue measure.
Define the continuous frames $F$ and $G$ for $\mathbb{R}^2$ with respect to $(B_{\mathbb{R}^2},\lambda)$ by
\[
F(\omega)=\left\{\begin{array}{ll}
\frac{1}{\sqrt{\lambda(B_1)}}e_1,\hspace{.5cm}\omega\in B_1,\cr
\frac{1}{\sqrt{\lambda(B_2)}}e_2,\hspace{.5cm}\omega\in B_2,\cr
0,\hspace{1.85cm}\omega\in B_3,\cr
\end{array}\right.
\]
and
\[
G(\omega)=\left\{\begin{array}{ll}
\frac{\varepsilon}{\sqrt{\lambda(B_1)}}e_1,\hspace{1.05cm}\omega\in B_1,\cr
0,\hspace{2.4cm}\omega\in B_2,\cr
\frac{1}{\sqrt{\lambda(B_3)}}e_2,\hspace{1.06cm}\omega\in B_3,\cr
\end{array}\right.
\]
where $\{B_1,B_2,B_3\}$ is a partition of $B_{\mathbb{R}^2}$.
For all $x\in\mathbb{R}^2$
\begin{align*}
(T_GT_F^*)(x)
& = \int_{B_{\mathbb{R}^2}}(T_F^*x)(\omega)G(\omega)d\lambda(\omega)\\
& = \left(\int_{B_1}+\int_{B_2}+\int_{B_3}\right)(T_F^*x)(\omega)G(\omega)d\lambda(\omega)\\
& = \int_{B_1}\langle x,\frac{\varepsilon}{\sqrt{\lambda(B_1)}}e_1\rangle\frac{1}{\sqrt{\lambda(B_1)}}e_1d\lambda(\omega)+0+0\\
& = \varepsilon\langle x,e_1\rangle e_1.
\end{align*}
Thus
$$\|x-(T_GT_F^*)(x)\|^2=(1-\varepsilon)^2|\langle x,e_1\rangle|^2+|\langle x,e_2\rangle|^2 < \|x\|^2.$$
Therefore, $F$ and $G$ are approximately dual continuous frames for $\mathbb{R}^2$.
\end{exa}
The following theorem shows that a Bessel mapping $F$
is a continuous frame for $\h$ with respect to $(\Omega,\mu)$,
if there exists a Bessel mapping $G$ such that
$F$ and $G$ are approximately dual frames.
\begin{thm}
If $F$ and $G$ are approximately dual continuous frames,
then  both $F$ and $G$ are continuous frames for $\h$ with respect to $(\Omega,\mu)$.
\end{thm}
\begin{proof}
Since $\|I_{\h}-T_GT_F^*\|<1$, the operator
$T_GT_F^*$ is an invertible operator on $\h$ and
 $$\|(T_GT_F^*)^{-1}\|\leq\frac{1}{1-\|I_{\h}-T_GT_F^*\|}.$$
Denote by $B_F$ and $B_G$ the Bessel constants of $F$ and $G$, respectively.
For all $f\in\h$
\begin{align*}
\|f\|
& \leq\|(T_GT_F^*)^{-1}\|\|T_GT_F^*f\|\\
& \leq\frac{1}{1-\|I_{\h}-T_GT_F^*\|}\sup_{\|g\|=1}|\langle T_GT_F^*f,g\rangle|\\
& \leq\frac{1}{1-\|I_{\h}-T_GT_F^*\|}\sup_{\|g\|=1}\int_{\Omega}|\langle f,F(\omega) \rangle\langle G(\omega),g \rangle|d\mu(\omega)\\
& \leq\frac{1}{1-\|I_{\h}-T_GT_F^*\|}\sup_{\|g\|=1}\left(\int_{\Omega}|\langle f,F(\omega) \rangle|^2 d\mu(\omega)\right)^{\frac{1}{2}}\left(\int_{\Omega}|\langle G(\omega),g\rangle|^2 d\mu(\omega)\right)^{\frac{1}{2}}\\
& \leq\frac{1}{1-\|I_{\h}-T_GT_F^*\|}\sup_{\|g\|=1}\left(\int_{\Omega}|\langle f,F(\omega) \rangle|^2 d\mu(\omega)\right)^{\frac{1}{2}}\sqrt{B_G}~\|g\|\\
& =\frac{\sqrt{B_G}}{1-\|I_{\h}-T_GT_F^*\|}\left(\int_{\Omega}|\langle f,F(\omega)\rangle|^2 d\mu(\omega)\right)^{\frac{1}{2}}.
\end{align*}
Hence $F$ is a continuous frame with lower bound $B_G^{-1}(1-\|I_{\h}-T_GT_F^*\|)^2$.
Similary, $G$ is a continuous frame with lower bound $B_F^{-1}(1-\|I_{\h}-T_GT_F^*\|)^2$.
\end{proof}
The following theorem shows that the sum of two approximately dual continuous frames is a continuous frame.
\begin{thm}
If $F$ and $G$ are approximately dual continuous frames,
then $F+G$ is a continuous frame for $\h$ with respect to $(\Omega,\mu)$.
\end{thm}
\begin{proof}
Suppose $T_F$ and $T_G$ be synthesis operators of $F$ and $G$ (respectively)
and $A_F, B_F, A_G$ and $B_G$ be lower and upper bounds of $F$ and $G$, respectively.
It is clear that $F+G$ is a Bessel mapping with Bessel constant $B_F+B_G$.
Since  $\|I_{\h}-T_GT_F^*\|<1$ or $\|I_{\h}-T_FT_G^*\|<1$,
then $\|2I_{\h}-(T_GT_F^*+T_FT_G^*)\|<2$,
and as $T_GT_F^*+T_FT_G^*$ is self-adjoint, then by Lemma 2.2.2 in \cite{mur},
$T_GT_F^*+T_FT_G^*$ is a positive operator.
For each $f\in\h$, we have
\begin{align*}
\int_{\Omega}|\langle f,(F+G)(\omega) \rangle|^2 d\mu(\omega)
& =  \int_{\Omega}|\langle f,F(\omega) \rangle|^2 d\mu(\omega)+\langle (T_GT_F^*+T_FT_G^*)f,f\rangle\\
&\hspace{.5cm}+\int_{\Omega}|\langle f,G(\omega)\rangle|^2 d\mu(\omega)\\
& \geq \int_{\Omega}|\langle f,F(\omega) \rangle|^2 d\mu(\omega)+\int_{\Omega}|\langle f,G(\omega)\rangle|^2 d\mu(\omega)\\
& \geq (A_F+A_G)\|f\|^2,
\end{align*}
i.e., the lower bound condition holds.
\end{proof}
It is clear that if $(F,G)$ is a dual pair for $\h$,
then $F$ and $G$ are approximately dual continuous frames,
but its converse isn't true in general.
The following theorem shows that
any approximate dual continuous frames is a kind of dual pair (as special sense).
\begin{thm}
If $F$ and $G$ are approximately dual continuous frames,
then there exists an invertible operator  $\Theta:\h\to\h$ such that
$(\Theta^* F,G)$ is a dual pair for $\h$.
\end{thm}
\begin{proof}
Since $\|I_{\h}-T_GT_F^*\|<1$, then $T_GT_F^*$ is an invertible operator on $\h$.
For each $f,g\in\h$
\begin{align*}
\langle f,g\rangle
& = \langle(T_GT_F^*)(T_GT_F^*)^{-1}f,g\rangle\\
& = \int_{\Omega}\langle f,((T_GT_F^*)^{-1})^*F(\omega)\rangle\langle G(\omega),g\rangle d\mu(\omega).\\
\end{align*}
The result follows by putting $\Theta=(T_GT_F^*)^{-1}$.
\end{proof}
\begin{thm}
Let $(F,G)$ be a dual pair for $\h$ and
let $U$ and $V$ be two bounded operators on $\h$ such that $\|I_{\h}-VU^*\|<1$.
Then $UF$ and $VG$ are approximately dual continuous frames.
\end{thm}
\begin{proof}
Since $UF$ and $VG$ are  Bessel mappings with synthesis operators $T_{UF}=UT_F$ and $T_{VG}=VT_G$ ( resp.),
so we have
$$\|I_{\h}-T_{VG}T_{UF}^*\|=\|I_{\h}-VT_GT_F^*U^*\|=\|I_{\h}-VU^*\|<1.$$
\end{proof}
\begin{cor}
If $(F,G)$ is a dual pair for $\h$ and $U$ is a unitary operator on $\h$,
then $(UF,UG)$ is approximately dual continuous frames.
\end{cor}
We proceed this section with the following result which
gives a sufficient and necessary condition for two continuous frames $F$ and $G$ under which
 they are approximately continuous frames.
To this end, recall that every bounded and positive operator $U:\h\to\h$
has a unique bounded and positive square root $U^{\frac{1}{2}}$.
Moreover, if the operator $U$ is self-adjoint (resp. invertible),
then $U^{\frac{1}{2}}$ is also self-adjoint (resp. invertible),
see A.6.7 of \cite{chri}.
\begin{thm}\label{4.8}
Let $F$ be continuous frame and $G$ a Bessel mapping for $\h$ with upper bounds $B_F$ and $B_G$, respectively.
Then $F$ and $G$ are approximately dual continuous frames if and only if
there exists a bounded operator $D\in\mathcal{B}(\h)$ such that
$$T_FT_G^*=S_F^{\frac{1}{2}}D,~~~DD^*\leq B_GI_{\h},~~~\|I_{\h}-S_F^{\frac{1}{2}}D\|<1.$$
\end{thm}
\begin{proof}
Since $F$ is a continuous frame for $\h$, then $S_F$ is a bounded and positive operator.
Hence it has a unique bounded square root $S_F^{\frac{1}{2}}$. The proof of the "only if" part is trivial.
To prove "if" part, suppose that $F$ and $G$ are approximately dual continuous frames.
For each $f\in\h$, we have
\begin{align*}
\|T_GT_F^*f\|
& = \sup_{\|g\|=1}|\langle T_GT_F^*f,g\rangle|\\
& = \sup_{\|g\|=1}|\int_{\Omega}\langle f,F(\omega)\rangle\langle G(\omega),g\rangle d\mu(\omega)|\\
& \leq \sup_{\|g\|=1}\left(\int_{\Omega}|\langle f,F(\omega)\rangle|^2 d\mu(\omega)\right)^{\frac{1}{2}}
\left(\int_{\Omega}|\langle g,G(\omega)\rangle|^2 d\mu(\omega)\right)^{\frac{1}{2}}\\
& \leq\sqrt{B_G}\left(\int_{\Omega}|\langle f,F(\omega)\rangle|^2 d\mu(\omega)\right)^{\frac{1}{2}}.
\end{align*}
Therefore, we have
$$\langle(T_FT_G^*)(T_FT_G^*)^*f,f\rangle\leq B_G\langle S_Ff,f\rangle,~~~f\in\h.$$
Thus for all $f\in\h$
$$(T_FT_G^*)(T_FT_G^*)^*\leq B_GS_F^{\frac{1}{2}}S_F^{\frac{1}{2}}.$$
The above inequality results:
\begin{enumerate}
\item According to Theorem 1 of \cite{dog}, there exists a bounded operator $D\in B(\h)$ such that $T_FT_G^*=S_F^{\frac{1}{2}}D$.
\item Theorem 2.2.5 in \cite{mur} implies that $DD^*\leq B_GI_{\h}$.
\end{enumerate}
\end{proof}
Now, we can prove the following theorems.
\begin{thm}\label{4.9}
Let $F$ be a continuous frame for $\h$ with respect to $(\Omega,\mu)$.
Then $F$ and the Bessel mapping $G$ are approximately dual continuous frames if and only if
$$G=D^*S_F^{-\frac{1}{2}}F+K,$$
where $D$ is a bounded operator on $\h$ for which $\|I_{h}-S_F^{\frac{1}{2}}D\|<1$
and $K$ is a Bessel mapping with property $T_FT_K^*=0$.
\end{thm}
\begin{proof}
First, we put $G=D^*S_F^{-\frac{1}{2}}F+K$, then for each $f\in\h$,
the $\omega\mapsto\langle f,G(\omega)\rangle$ is a measurable function.
Also, for each $f\in\h$, we have
$$\int_{\Omega}|\langle f,G(\omega)\rangle|^2d\mu(\omega)\leq(B_K+\|D\|^2)\|f\|^2.$$
Therefore, $G$ is a Bessel mapping with Bessel bound $B_G=B_K+\|D\|^2$.
Now, for all $f\in\h$ and $\omega\in\Omega$, we have
\begin{align*}
(T_G^*f)(\omega)
& = \langle f,D^*S_F^{-\frac{1}{2}}F(\omega)\rangle+\langle f,K(\omega)\rangle\\
& = \langle S_F^{-\frac{1}{2}}Df,F(\omega)\rangle+\langle f,K(\omega)\rangle\\
& = T_F^*(S_F^{-\frac{1}{2}}Df)(\omega)+(T_K^*f)(\omega)
\end{align*}
and
\begin{align*}
T_FT_G^*f
& = T_FT_F^*(S_F^{-\frac{1}{2}}Df)+T_FT_K^*f\\
& =  S_F^{\frac{1}{2}}Df.\\
\end{align*}
Hence $T_FT_G^*=S_F^{\frac{1}{2}}D$.
Moreover,
$$\langle DD^*f,f\rangle=\langle D^*f,D^*f\rangle=\|D^*f\|^2\leq\|D\|^2\|f\|^2\leq B_G\|f\|^2=B_G\langle f,f\rangle,$$
for all $f\in\h$.
Now, we use previous theorem to conclude that $F$ and $G$ are approximately dual continuous frames.\\
Conversely, let $F$ and $G$ are approximately dual continuous frames.
According to Theorem \ref{4.8},
there exists an invertible operator $D$ in $\mathcal{B}(\h)$ such that
$T_FT_G^*=S_F^{\frac{1}{2}}D$ and $\|I_{\h}-S_F^{\frac{1}{2}}D\|<1$.
Put $K=G-D^*S_F^{-\frac{1}{2}}F$.
It is clear that $K$ is a Bessel mapping with bound
$B_K=B_G+\|D\|^2$.
We have for all $f\in\h$ and all $\omega\in\Omega$,
\begin{align*}
(T_G^*-T_F^*S_F^{-\frac{1}{2}}Df)(\omega)
& = \langle f,G(\omega)\rangle-\langle S_F^{-\frac{1}{2}}Df,F(\omega)\rangle\\
& = \langle f,G(\omega)\rangle-\langle f,D^*S_F^{-\frac{1}{2}}F(\omega)\rangle\\
& = \langle f,G(\omega)-D^*S_F^{-\frac{1}{2}}F(\omega)\rangle\\
& = \langle f,K(\omega)\rangle\\
& = (T_K^*f)(\omega).
\end{align*}
Thus
$$T_FT_K^*=T_FT_G^*-T_FT_F^*S_F^{-\frac{1}{2}}D=T_FT_G^*-S_F^{\frac{1}{2}}D=0,$$
and this completes the proof.
\end{proof}
\begin{thm}
Let $F$ be a continuous frame for $\h$ with respect to $(\Omega,\mu)$.
Then $F$ and the Bessel mapping $G$ are approximately dual continuous frames if and only if
$$G=D^*S_F^{-\frac{1}{2}}F-F+S_FK,$$
where $D$ is a bounded operator on $\h$ for which $\|I_{\h}-S_F^{\frac{1}{2}}D\|<1$
and $K$ is a Bessel mapping such that $(F,K)$ is a dual pair for $\h$.
\end{thm}
\begin{proof}
Put $G=D^*S_F^{-\frac{1}{2}}F-F+S_FK$ is a Bessel mapping with Bessel constant $B_G=\|D\|^2+B_F+B_G$.
Since $T_G^*=T_F^*S_F^{-\frac{1}{2}}D-T_F^*+T_K^*S_F$,
then $T_FT_G^*=S_F^{\frac{1}{2}}D$.
Moreover,
$$DD^*\leq B_GI_{\h}.$$
Conversely, let $F$ and $G$ are approximately dual continuous frames.
According to Theorem \ref{4.8},
there exists an invertible operator $D$ in $\mathcal{B}(\h)$ such that
$T_FT_G^*=S_F^{\frac{1}{2}}D$ and $\|I_{\h}-S_F^{\frac{1}{2}}D\|<1$.
Put $K=S_F^{-1}G-S_F^{-1}D^*S_F^{-\frac{1}{2}}F-S_F^{-1}F$.
It is clear that $K$ is a Bessel mapping with bound
$B_K=(B_G+\|D\|^2+B_F)\|S_F^{-1}\|^2$.
Since
$$T_K^*=T_G^*S_F^{-1}-T_F^*S_F^{-\frac{1}{2}}DS_F^{-1}+T_F^*S_F^{-1},$$
then $T_FT_K^*=I_{\h}$
and this completes the proof.
\end{proof}

\subsection{On perturbation of continuous frames}
Perturbation theory is a very important concepts in several areas of mathematics.
It went back to classical perturbation results by Paley and Wienr in 1934.
The perturbations of discrete frames has been discussed in \cite{4}.
For continuous frames, it was studied in \cite{dr2,bal}.
In this subsection, we are giving some results on perturbation of continuous frames in point of view the duality notion.
\begin{thm}
Let $F$ be a Parseval continuous frame and $G$ be a Bessel mapping.
Assume that there exist constants $\lambda, \gamma\geq0$ such that
\begin{align*}
\|\int_{\Omega}\phi(\omega)(F(\omega)-G(\omega))d\mu(\omega)\| \leq
&  \lambda\|\int_{\Omega}\phi(\omega)F(\omega)d\mu(\omega)\|\\
& \hspace{.1cm}+\gamma\left(\int_{\Omega}|\phi(\omega)|^2d\mu(\omega)\right)^{\frac{1}{2}},
\end{align*}
for all $\phi\in L^2(\Omega)$.
If $\lambda+\gamma<1$,
then $(F,G)$ is an approximately dual continuous frames.
\end{thm}
\begin{proof}
For all $f\in\h$
\begin{align*}
\|f-(T_GT_F^*)(f)\|
& = \|\int_{\Omega}\langle f,F(\omega)\rangle(F(\omega)-G(\omega))d\mu(\omega)\|\\
& \leq \lambda\|\int_{\Omega}\langle f,F(\omega)\rangle F(\omega)d\mu(\omega)\|\\
& \hspace{.5cm}+\gamma\left(\int_{\Omega}|\langle f,F(\omega)\rangle|^2d\mu(\omega)\right)^{\frac{1}{2}}\\
& < \|f\|.
\end{align*}
Thus $\|I_{\h}-T_GT_F^*\|<1$.
Consequently, $F$ and $G$ are approximately dual continuous frames.
\end{proof}
\begin{thm}
Let $F$, $G$ and $K$ be Bessel mappings and $B_G$ be the Bessel constant of $G$.
Assume that there exists $\lambda>0$ such that  $\lambda B_G<1$ and for all $f\in\h$
$$\int_{\Omega}|\langle f,F(\omega)-K(\omega)\rangle|^2d\mu(\omega)\leq \lambda\|f\|^2.$$
If $(F,G)$ is a dual pair for $\h$, then $G$ and $K$ are approximately dual continuous frames.
\end{thm}
\begin{proof}
Since for any $f\in\h$, $(T_F^*f-T_K^*f)(\omega)=\langle f, F(\omega)-K(\omega)\rangle$,
then
$$\|(T_F^*-T_K^*)f\|^2=\int_{\Omega}|\langle f,F(\omega)-K(\omega)\rangle|^2\leq \lambda\|f\|^2$$
and consequently, $\|T_F^*-T_H^*\|\leq\sqrt{\lambda}$.
Now we have
$$\|I_{\h}-T_GT_K^*\|=\|T_G(T_F^*-T_K^*)\|\leq\|T_G\|\|T_F^*-T_K^*\|\leq\sqrt{\lambda B_G}<1.$$
\end{proof}
\begin{thm}
Let $(F,G)$ be a dual pair for $\h$ and $K:\Omega\to\h$ be a Bessel mapping.
Assume that there exist constants $\lambda, \gamma\geq0$ such that
\begin{align*}
\|\int_{\Omega}\phi(\omega)(F(\omega)-K(\omega))d\mu(\omega)\| \leq
& \lambda\|\int_{\Omega}\phi(\omega)F(\omega)d\mu(\omega)\|\\
& \hspace{.1cm}+\gamma\left(\int_{\Omega}|\phi(\omega)|^2d\mu(\omega)\right)^{\frac{1}{2}},
\end{align*}
for all $\phi\in L^2(\Omega)$.
If $\lambda+\gamma\sqrt{B_G}<1$,
Then $G$ and $K$ are approximately dual continuous frames,
where $B_G$ is Bessel constant of $G$.
\end{thm}
\begin{proof}
For all $f\in\h$
\begin{align*}
\|f-(T_KT_G^*)(f)\|
& = \|\int_{\Omega}\langle f,G(\omega)\rangle(F(\omega)-K(\omega))d\mu(\omega)\|\\
& \leq \lambda\|\int_{\Omega}\langle f,G(\omega)\rangle F(\omega)d\mu(\omega)\|\\
& \hspace{.5cm}+\gamma\left(\int_{\Omega}|\langle f,G(\omega)\rangle|^2d\mu(\omega)\right)^{\frac{1}{2}}\\
& \leq (\lambda+\gamma\sqrt{B_G})\|f\|.
\end{align*}
Thus $\|I_{\h}-T_KT_G^*\|<1$.
\end{proof}
\textbf{Acknowledgement:}
The authors would like to thank the referee(s) for their useful suggestions and comments.  


\begin{thebibliography}{99}

\bibitem{ali}
 S. T. Ali, J. P. Antoine and J. P. Gazeau,
{\it Continuous frames in Hilbert spaces},
Annals of physics,
\textbf{222}, 1-37 (1993).

\bibitem{arefi}
A. A. Arefijamaal, R. A. Kamyabi Gol, R. Raisi Tousi and N. Tavallaei,
{\em A new approach to continuous Riesz bases},
J. Sci. Iran,
\textbf{24(1)}, 63-69 (2013).

\bibitem{Balazs}
P. Balazs, J. P. Antoine and A. Grybos,
{\em Weighted and controlled frames},
Int. J. Wavelets Multiresolut. Inf. Process. \textbf{8(1)}, 109-132 (2010).

\bibitem{bal}
P. Balazs, D. Bayer and A. Rahimi,
{\em Multipliers for continuous frames in Hilbert spaces},
J. Phys. A: Math. Theor. \textbf{45}, 244023 (20pp) (2012).

\bibitem{balfe}
P. Balazs, H. G. Feichtinger, M. Hampejs and G. Kracher,
{\em Double preconditioning for Gabor frames},
IEEE Trans. Signal Process. \textbf{54}, 4597-4610 (2006).

\bibitem{4}
P. G. Cazassa, O. Christensen,
{\em Perturbation of operators and applications to frame theory},
J. Fourier Anal. Appl. \textbf{3} (1997), 543--557.

\bibitem{6}
P. G. Casazza and G. Kutyniok,
{\em  Frames of subspaces, Wavelets, Frames and Operator Theory},
(College Park, MD, 2003),
Contempt. Math. \textbf{345}, Amer. Math. Soc. Providence, RI. (2004), 87-113.

\bibitem{chri}
O. Christensen,
{\em An Introduction to Frames and Riesz Bases},
Birkhauser, Boston 2003.

\bibitem{olela}
O. Christensen and R. S. Laugesen,
{\em Approximately dual frames in Hilbert spaces and applications to Gabor frames},
Sampling theory in signal and image processing, Vol. \textbf{9}, No. 1-3, 77-89 (2010).


\bibitem{da}
I. Daubechies and B. Han,
{\em The canonical dual frame of a wavelet frame},
Applied and Comput. Harm. Anal. \textbf{12}, 269-285 (2002).

\bibitem{dog}
 R. G. Douglas,
{\em On majorization, factorization and range inclusion of operators on Hilbert spaces},
Proc. Amer. Math. Soc. 17 (1966), 413-415.

\bibitem{duf}
R. J. Duffin, A. C. Schaeffer,
{\em A class of nonharmonic Fourier series},
Trans. Amer. Math. Soc. \textbf{72}, 341-366 (1952).

\bibitem{fegr}
H. G. Feichtinger and K. Grochenig,
{\em Banach spaces related to integrable group representations and their atomic decomposition},
I. J. Funct. Anal. \textbf{86}, 307-340 (1989).

\bibitem{feka}
H. G. Feichtinger and N. Kaiblinger,
{\em Varing the time-frequency lattice of Gabor frames},
Trans. Amer. Math. Soc. \textbf{356}, 2001-2003 (2004).

\bibitem{Gavruta}
L. G\v{a}vruta,
{\em  Frames for operators},
Appl. Comput. Harmon. Anal. \textbf{32}, 139-144 (2012).

\bibitem{han}
J. P. Gabardo and D. Han,
{\em Frame associated with measurable spaces},
Adv. Comp. Math. \textbf{18}, no. 3, 127-147 (2003).

\bibitem{hol}
M. Holschneider,
{\em Wavelet. An analysis tool},
Oxford Mathematical Monographs. Oxford Science Publications. The Clarendon Press, Oxford University Press, New York, 1995.

\bibitem{kai}
G. Kaiser,
{\em A friendly guide to wavelets},
Birkh\"{a}user, Boston, MA, 1994.

\bibitem{li}
S. Li,
{\em On general frame decompositions},
Numer. Funct. Anal. and Optimiz. \textbf{16}, Issue 9-10, 1181-1191 (1995).

\bibitem{liya}
S. Li and D. Yan,
{\em Frame fundamental sensor modeling and stability of one-sided frame perturbation},
Acta Applicandae Mathematicae \textbf{107}, 91-103 (2009).

 \bibitem{mur}
G. J. Murphy,
{\em $C^*-$algebra and operator theory},
San Diego: Academic Press, 1990.

\bibitem{18}
A. Rahimi and P. Balazs,
\textit{ Multipliers for p-Bessel sequences in Banach spaces},
Integr. Equ. Oper. Theory  \textbf{68}, no. 2, 193--205 (2010).

\bibitem{rafe}
A. Rahimi and A. Fereydooni,
{\em Controlled G-frames and Their G-multipliers in Hilbert spaces},
An. St. Univ. Ovidius Constanta, Vol. \textbf{21}(2),  223--236 (2013).

\bibitem{dr2}
A. Rahimi, A. Najati and Y. N. Dehghan,
{\em Continuous frames in Hilbert spaces},
Method of Functional Analysis and Topology,
Vol. \textbf{12} No. 2, 170-182 (2006).

\bibitem{sun}
W. Sun,
{\em G-frames and G-Riesz bases},
J. Math. Anal. Appl.
\textbf{322}(1), 437-452 (2006).


\end{thebibliography}
\end{document}